\documentclass[final]{siamltexmm}
\usepackage{mymacros}
\usepackage{verbatim}

\title{Sublinear randomized algorithms for skeleton decompositions\thanks{Dated April  2012.}}
\author{Jiawei Chiu \thanks{Corresponding author. \email{jiawei@mit.edu}. Department of Mathematics, MIT, Cambridge, MA 02139, USA.}\and Laurent Demanet \thanks{Department of Mathematics, MIT, Cambridge, MA 02139, USA.}}

\begin{document}
\maketitle
\newcommand{\slugmaster}{%
}

\begin{abstract}
Let $A$ be a $n$ by $n$ matrix. A skeleton decomposition is any factorization of the form $CUR$ where $C$ comprises columns of $A$, and $R$ comprises rows of $A$. In this paper, we consider uniformly sampling $\l\simeq k \log n$ rows and columns to produce a skeleton decomposition. The algorithm runs in $O(\l^3)$ time, and has the following error guarantee. Let $\norm{\cdot}$ denote the $2$-norm. Suppose $A\simeq X B Y^T$ where $X,Y$ each have $k$ orthonormal columns. Assuming that $X,Y$ are incoherent,  we show that with high probability, the approximation error $\norm{A-CUR}$ will scale with $(n/\l)\norm{A-X B Y^T}$ or better. A key step in this algorithm involves regularization. This step is crucial for a nonsymmetric $A$ as empirical results suggest. Finally, we use our proof framework to analyze two existing algorithms in an intuitive way.
\end{abstract}

\pagestyle{myheadings}
\thispagestyle{plain}

\section*{\small Acknowledgments}
{\footnotesize JC is supported by the A*STAR fellowship from Singapore. LD is supported in part by the National Science Foundation and the Alfred P. Sloan foundation. We would also like to thank the referees for making many valuable suggestions.}

\section{Introduction}
\subsection{Skeleton decompositions}
This paper is concerned with the decomposition known as the matrix skeleton\footnote{The term ``skeleton'' may refer to other factorizations.

Instead of $A\simeq A_{C:} Z A_{:R}$, we can have $A \simeq Z_1 A_{RC} Z_2$ where $Z_1,Z_2$ are arbitrary $m\times k$ and $k\times n$ matrices \cite{cheng2005compression}. As $O(mk+nk)$ space is needed to store $Z_1,Z_2$, this representation does not seem as appealing. Nevertheless, it is numerically more stable and has found several applications \cite{Ho2011}.

When $A=M B N$ where $M,B,N$ are $n\times n$ matrices, we can approximate $M$ as $M_{:C} P$, $N$ as $D N_{R:}$, where $M_C$ has $k$ columns of $M$ and $N_R$ has $k$ rows of $N$. Thus, $A\simeq M_C (P B D) N_R$, effectively replacing $B$ with the $k\times k$ matrix $\widetilde{B}:=P B D$. Bremer calls $\widetilde{B}$ a skeleton and uses it to approximate scattering matrices \cite{bremer2011fast}.}, pseudo-skeleton \cite{goreinov1997theory}, or CUR factorization  \cite{mahoney2009cur,drineas:844}.

For the rest of the paper, we adopt the following Matlab-friendly notation: given $A\in \cplexes^{m\times n}$, $A_{:C}$ will denote the restriction of $A$ to columns indexed by $C$, and $A_{R:}$ will denote the restriction of $A$ to rows indexed by $R$. A skeleton decomposition of $A$ is any factorization of the form
$$A_{:C} Z A_{R:} \mbox{ for some } Z\in \cplexes^{k\times k}.$$

In general, a rank-$k$ approximation of $A$ takes up $O((m+n)k)$ space. The major advantage of the skeleton decomposition is that it may require only $O(k^2)$ space. Consider the case where $A$'s entries can be specified by a function in closed form. Then the skeleton decomposition is fully specified by $Z\in \cplexes^{k\times k},$ as well as the two index sets $C$ and $R$. In addition, row and columns from the original matrix may carry more physical significance than linear combinations thereof.

There are important examples where the full matrix itself is not low rank but can be partitioned into blocks each of which has low numerical rank. One example is the Green's function of elliptic operators with mild regularity conditions \cite{bebendorf2003existence}. Another example is the amplitude factor in certain Fourier integral operators and wave propagators \cite{candes:fio, demanet:wavefio}. Algorithms that compute good skeleton representations can be used to compress such matrices.

\subsection{Overview}
\begin{figure}[t] \begin{center}\medskip\fbox {\parbox{0.9\linewidth}{
\begin{algorithm}\label{alg:fullspread-k3} \textup{ Skeleton via uniform sampling, $\widetilde{O}(k^3)$ algorithm. \\
\underline{Input}: A matrix $A\in \cplexes^{m\times n}$, and user-defined parameters $\l$ and $\delta$.\\
\underline{Output}: A skeleton representation with $\l$ rows and $\l$ columns.\\
\underline{Steps}:
\begin{enumerate}
\item Uniformly and independently select $\l$ columns to form $A_{:C}$.
\item Uniformly and independently select $\l$ rows to form $A_{R:}$.
\item Compute the SVD of $A_{RC}$. Remove all the components with singular values less than $\delta$. Treat this as a perturbation $E$ of $A$ such that $\norm{E_{RC}}\leq \delta$ and $\norm{(A_{RC}+E_{RC})^{+}}\leq \delta^{-1}$.
\item Compute $Z=(A_{RC}+E_{RC})^+$. (In Matlab, \texttt{Z=pinv(A(R,C),delta)}.)
\item Output $A_{:C} Z A_{R:}$.
\end{enumerate}
} \end{algorithm} 
}}\medskip \end{center}
\end{figure}

In this paper, we consider uniformly sampling $\l\simeq k \log \max(m,n)$ rows and columns of $A$ to build a skeleton decomposition. After obtaining $A_{:C},A_{R:}$ this way, we compute the middle matrix $Z$ as the pseudoinverse of $A_{RC}$ with \emph{thresholding} or regularization. See Algorithm \ref{alg:fullspread-k3} for more details. Suppose $A \simeq X_1 A_{11} Y_1^T$ where $X_1,Y_1$ have $k$ orthonormal columns. Assuming that $X_1,Y_1$ are \emph{incoherent}, we show in Theorem \ref{thm:fullspread-k3} that the $2$-norm approximation error $\norm{A-A_{:C} Z A_{R:}}$ is bounded by $O(\norm{A-X_1 A_{11} Y_1^T}(mn)^{1/2}/\l) $ with high probability. We believe that this is the first known relative error bound in the $2$-norm for a nonsymmetric $A$.

The idea of uniformly sampling rows and columns to build a skeleton is not new. In particular, for the case where $A$ is symmetric, this technique may be known as the Nystrom method\footnote{In machine learning, the Nystrom method can be used to approximate kernel matrices of support vector machines, or the Laplacian of affinity graphs, for instance.}. The skeleton used is $A_{:C} A_{CC}^+ A_{C:}$, which is symmetric, and its $2$-norm error is recently analyzed by Talwalkar \cite{talwalkar2010matrix} and Gittens \cite{gittens2011spectral}. Both papers make the same assumption that $X_1,Y_1$ are incoherent. In fact, Gittens arrives at the same relative error bound as us.

Nonetheless, our results are more general. They apply to nonsymmetric matrices that are low rank in a broader sense. Specifically, when we write $A\simeq X_1 A_{11} Y_1^T$, $A_{11}$ is not necessarily diagonal and $X_1,Y_1$ are not necessarily the singular vectors of $A$. This relaxes the incoherence requirement on $X_1,Y_1$. Furthermore, in the physical sciences, it is not uncommon to work with linear operators that are known a priori to be \emph{almost} diagonalized by the Fourier basis or related bases in harmonic analysis. These bases are often incoherent. One example is an integral operator with a smooth kernel. See Section \ref{sec:numerical} for more details.

The drawback of our results is that it requires us to set an appropriate regularization parameter in advance. Unfortunately, there is no known way of estimating it fast, and this regularization step cannot be skipped. In Section \ref{sec:toy}, we illustrate with a numerical example that without the regularization, the approximation error can blow up in a way predicted by Theorem \ref{thm:fullspread-k3}.

Finally, we also establish error bounds for two other algorithms. We shall postpone the details.

\subsection{Some previous work}

Before presenting our main results, we like to provide a sample of some previous work. The well-informed reader may want to move on.

The precursor to the skeleton decomposition is the interpolative decomposition \cite{cheng2005compression}, also called the column subset selection problem \cite{frieze2004fast}. An interpolative decomposition of $A$ is the factorization $A_{:C} D$ for some $D$. The earliest work on this topic is probably the pivoted QR algorithm by Businger, Golub \cite{businger1965linear} in 1965. In 1987, Chan \cite{chan1987rank} introduced the Rank Revealing QR (RRQR); a sequence of improved algorithms and bounds followed \cite{foster1986rank, hong1992rank,pan1999bounds,gu1996efficient}. RRQR algorithms can also be used to compute skeletons \cite{liberty2007randomized}.

An early result due to Goreinov et al. \cite{goreinov1997theory} says that for any $A\in \cplexes^{m\times n}$, there exists a skeleton $A_{:C} Z A_{R:}$ such that in the $2$-norm, $\norm{A-A_{:C} Z A_{R:}} = O(\sqrt{k}(\sqrt{m}+\sqrt{n})\sigma_{k+1}(A))$. Although the proof is constructive, it requires computing the SVD of $A$, which is much more costly than the algorithms considered in this paper. An useful idea here is to maximize the volume or determinant of submatrices. This idea is also used to study RRQRs \cite{chandrasekaran1994rank}, and may date back to the proof of Auerbach's theorem \cite{ruston1962auerbach}, interpolating projections \cite{price1970minimal} etc.

One way of building a skeleton is to iteratively select good rows and columns based on the residual matrix. This is known as Cross Approximation. As processing the entire residual matrix is not practical, there are faster algorithms that operate on only a small part of the residual, e.g., Adaptive Cross Approximation \cite{bebendorf2000approximation, bebendorf2006accelerating}, Incomplete Cross Approximation \cite{tyrtyshnikov2000incomplete}. These methods are used to compress off-diagonal blocks of ``asymptotically smooth'' kernels. The results of this paper will also apply to such smooth kernels. See Section \ref{sec:smoothkernel}.

Another way of building a skeleton is to do random sampling, which may date back to \cite{frieze2004fast}. One way of sampling is called ``subspace sampling'' \cite{mahoney2009cur,drineas:844} by Drineas et al. If we assume that the top $k$ singular vectors are incoherent, then a result due to Rudelson, Vershynin \cite{rudelson1999random} implies that \emph{uniform sampling} of rows and columns, a special case of ``subspace sampling'', will produce a good skeleton representation $A_{:C} (A_{:C}^+ A A_{R:}^+) A_{R:}$. However, it is not clear how the middle matrix $A_{:C}^+ A A_{R:}^+$ can be computed in sublinear time. In this paper, the skeleton we analyze resembles $A_{:C} A_{RC}^+ A_{R:}$, which can be computed in $\widetilde{O}(k^3)$ time.

\subsection{Preliminaries}

The matrices we consider in this paper take the form
\begin{equation}\label{eq:A}
A=(X_1 \quad X_2) \twotwo{{A}_{11}}{{A}_{12}}{{A}_{21}}{{A}_{22}}\binom{Y_1^T}{Y_2^T},
\end{equation}
where $X=(X_1 \quad X_2)$ and $Y=(Y_1 \quad Y_2)$ are unitary matrices, with columns being ``\emph{spread}'', and the blocks $A_{12}, A_{21}$ and $A_{22}$ are in some sense \emph{small}.  By ``spread'', we mean $\widetilde{O}(1)$-coherent.

\begin{definition}\label{def:spread}
Let $X\in \cplexes^{n \times k}$ be a matrix with $k$ orthonormal columns. Denote $\norm{X}_{\max}=\max_{ij}\abs{X_{ij}}$. We say $X$ is $\mu$-coherent if $\norm{X}_{\max}\leq (\mu/n)^{1/2}$.
\end{definition}

This notion is well-known in compressed sensing \cite{candes2007sparsity} and matrix completion \cite{candes2009exact, negahban2010restricted}.

To formalize ``small", let $\Delta_k:=\twotwo{0}{{A}_{12}}{{A}_{21}}{{A}_{22}}$, and consider that
\begin{equation}\label{eq:eps}
\eps_k:= \norm{\Delta_k} \mbox{ is small}.
\end{equation}
By tolerating an $\eps_k$ error in the $2$-norm, $A$ can be represented using only $O(k^2)$ data. Note that $\eps_k$ is equivalent to $\max(\norm{X_2^T A},\norm{A Y_2})$ up to constants. To prevent clutter, we have suppressed the dependence on $k$ from the definitions of $X_1,Y_1,A_{11}, A_{12}$ etc.

If (\ref{eq:A}) is the SVD of $A$, then $\eps_k=\sigma_{k+1}(A)$. It is good to keep this example in mind as it simplifies many formulas that we see later.

An alternative to $\eps_k$ is
\begin{equation}\label{eq:eps1}
\eps'_k:= \sum_{i=1}^m \sum_{j=1}^n \abs{(\Delta_k)_{ij}}.
\end{equation}
In other words, $\eps'_k$ is the $\l^1$ norm of $\Delta_k$ reshaped into a vector. We know $\eps_k\leq \eps'_k \leq mn \eps_k$. The reason for introducing $\eps'_k$ is that it is common for $(\Delta_k)_{ij}$ to decay rapidly such that $\eps'_k \ll mn \eps_k$. For such scenarios, the $2$-norm error may grow much slower with $m,n$.

\subsection{Main result}\label{sec:mainresults}
Our main contribution is the error analysis of Algorithm \ref{alg:fullspread-k3}. The problem with the algorithm is that it does not always work. For example, if $A=X_1 A_{11} Y_1^T$ and $X_1=\binom{I_{k\times k}}{0}$, then $A_{R:}$ is going to be zero most of the time, and so is $A_{:C}Z A_{R:}$. Hence, it makes sense that we want $X_{1,R:}$ to be ``as nonsingular as possible'' so that little information is lost.  In particular, it is well-known that if $X_1,Y_1$ are $\widetilde{O}(1)$-coherent, i.e., \emph{spread}, then sampling $\ell=\widetilde{O}(k)$ rows will lead to $X_{1,R:},Y_{1,C:}$ being well-conditioned\footnote{The requirement that $\norm{Y_1}_{\max}=\widetilde{O}(n^{-1/2})$ can be relaxed in at least two ways. First, all we need is that $\max_{i} (\sum_j \abs{(Y_1)_{ij}}^2)^{1/2}\leq (\mu k/n)^{1/2}$ which would allow a few entries of every row of $Y_1$ to be big. Second, if only a few rows of $Y$ violate this condition, we are still fine as explained in \cite{avron2010blendenpik}.}.

Here is our main result. It is proved in Section \ref{sec:theory}.
\begin{theorem}\label{thm:fullspread-k3}
Let $A$ be given by (\ref{eq:A}) for some $k > 0$. Assume $m\geq n$ and $X_1,Y_1$ are $\mu$-coherent. Recall the definitions of $\eps_k, \eps'_k$ in  (\ref{eq:eps}) and (\ref{eq:eps1}). Let $\l\geq 10 \mu k \log m$ and $\lambda = \frac{(mn)^{1/2}}{\l}$. Then with probability at least $1-4km^{-2}$, Algorithm \ref{alg:fullspread-k3} returns a skeleton that satisfies
\begin{equation}\label{eq:basicbound}
\norm{A-A_{:C} Z A_{R:}} =O(\lambda \delta+\lambda \eps_k+\eps_k^2 \lambda/\delta).
\end{equation}
If the entire $X$ and $Y$ are $\mu$-coherent, then with probability at least $1-4m^{-1}$,
\begin{equation}\label{eq:extrabound}
\norm{A-A_{:C} Z A_{R:}} = O(\lambda \delta+ \eps'_k + {\eps'_k}^2/(\lambda \delta)).
\end{equation}
\end{theorem}

Minimize the RHS of (\ref{eq:basicbound}) and (\ref{eq:extrabound}) with respect to $\delta$. For (\ref{eq:basicbound}), pick $\delta=\Theta(\eps_k)$ so that
\begin{equation}\label{eq:basicconsequence}
\norm{A - A_{:C} Z A_{R:}}=O(\eps_k \lambda)=O( \eps_k (mn)^{1/2}/\l).
\end{equation}
For (\ref{eq:extrabound}), pick $\delta=\Theta(\eps'_k/\lambda)$ so that 
\begin{equation}\label{eq:extraconsequence}
\norm{A - A_{:C} Z A_{R:}}=O( \eps'_k).
\end{equation}

Here are some possible scenarios where $\eps'_k=o(\eps_k \lambda)$ and (\ref{eq:extraconsequence}) is much better than (\ref{eq:basicconsequence}):
\begin{itemize}
\item The entries of $\Delta_k$ \emph{decay exponentially} or there are only $O(1)$ nonzero entries as $m,n$ increases. Then  $\eps'_k=\Theta(\eps_k)$.
\item Say $n=m$ and (\ref{eq:A}) is the SVD of $A$. Suppose the singular values decay as $m^{-1/2}$. Then $\eps'_k = O(\eps_k  m^{1/2})$. \end{itemize}

One important question remains: what is $\eps_k$? Unfortunately, we are not aware of any $\widetilde{O}(k^3)$ algorithm that can accurately estimate $\eps_k$. Here is one possible \emph{heuristic} for choosing $\delta$ for the case where (\ref{eq:A}) is the SVD. Imagine $A_{RC}\simeq X_{1,R:} {A}_{11} Y_{1,C:}^T$. As we will see, the singular values of $X_{1,R:},Y_{1,C:}$ are likely to be on the order of $(\l/m)^{1/2}, (\l/n)^{1/2}$. Therefore, we may pretend that $\eps_k = \sigma_{k+1}(A) \simeq \lambda \sigma_{k+1}(A_{RC})$.

Another approach is to begin with a big $\delta$, run the $\widetilde{O}(k^3)$ algorithm, check $\norm{A-A_{:C} Z A_{R:}}$, divide $\delta$ by two and repeat the whole process until the error does not improve. However, calculating $\norm{A-A_{:C} Z A_{R:}}$ is expensive and we will need other tricks. This is an open problem.

The $\widetilde{O}(k^3)$ algorithm is probably the fastest algorithm for computing skeleton representations that one can expect to have. If we do more work, what can we gain? In Section \ref{sec:extend}, we sketch two such algorithms. The first one samples $\l \simeq k \log m$ rows, columns, then reduce it to \emph{exactly} $k$ rows, columns using RRQR, with a $2$-norm error of $O(\eps_k (mk)^{1/2})$. It is  similar to what is done in \cite{boutsidis2009improved}. In the second algorithm, we uniformly sample $\l \simeq k \log m$ rows to get $A_{R:}$, then run RRQR on $A_{R:}$ to select $k$ columns of $A$. The overall error is $O(\eps_k (mn)^{1/2})$. This is similar to the algorithm proposed by Tygert, Rokhlin et al. \cite{liberty2007randomized,woolfe2008fast}.

Using the proof framework in Section \ref{sec:theory}, we will derive error estimates for the above two algorithms. Our results are not new, but they work for a more general model (\ref{eq:A}) and the proofs are better motivated. We will also compare these three algorithms in Section \ref{sec:summary}.

\subsection{More on incoherence}\label{sec:introsrft}
Haar unitary matrices are $\widetilde{O}(1)$-coherent \cite[Theorem VIII.1]{donoho2001uncertainty}, that is if we draw $X,Y$ uniformly from the special orthogonal group, it is very likely that they are incoherent. This means Algorithm \ref{alg:fullspread-k3} will work well with $\ell =\widetilde{O}(k)$.

If either $X$ or $Y$ is not $\widetilde{O}(1)$-coherent, we can use an {old} trick in compressed sensing \cite{avron2010blendenpik}: multiply them on the left by the unitary Fourier matrix with randomly rescaled columns. This has the effect of ``blending up'' the rows of $X,Y$. The following is an easy consequence of Hoeffding's inequality, and the proof is omitted.
\begin{proposition}\label{thm:smoother}
Let $X\in \cplexes^{n\times k}$ with orthonormal columns. Let $D=\diag(d_1,\ldots,d_n)$ where $d_1,\ldots,d_n$ are independent random variables such that $\E d_i=0$ and $\abs{d_{i}}=1$. Let $\mathcal{F}$ be the unitary Fourier matrix and  $\mu=\alpha \log n$ for some $\alpha>0$. Define $U:=\mathcal{F} D X$. Then $\norm{U}_{\max} \leq (\mu/n)^{1/2}$ with probability at least $1-2(nk)n^{-2 \alpha}$
\end{proposition}

In words, no matter what $X$ is, $U=\mathcal{F} D X$ would be $\widetilde{O}(1)$-coherent whp. Hence, we can write a simple wrapper around Algorithm \ref{alg:fullspread-k3}.
\begin{enumerate}
\item Let $B:=\mathcal{F} D_2 A D_1 \mathcal{F}^T$ where $D_1,D_2$ are diagonal matrices with independent entries that are $\pm 1$ with equal probability.
\item Feed $B$ to the $\widetilde{O}(k^3)$ algorithm and obtain $B \simeq B_{:C} Z B_{R:}$.
\item It follows that $A\simeq (A D_1 \mathcal{F}_{C:}^T)  Z (\mathcal{F}_{R:} D_2 A)$.
\end{enumerate}
The output is not a skeleton representation, but the amount of space needed is $O(n)+\widetilde{O}(k^2)$ which is still better than $O(nk)$. However, the algorithm may no longer run in sublinear time.

\section{Error estimates for $\widetilde{O}(k^3)$ algorithm}\label{sec:theory}
\subsection{Notation}
$S_C,S_R\in \cplexes^{n\times k}$ are both \emph{column} selector matrices. They are column subsets of permutation matrices. The subscripts ``$R:$'' and ``$:C$'' denote a row subset and a column subset respectively, e.g., $A_{R:}=S_R^T A$ and $A_{:C}=A S_C$, while $A_{RC}$ is a row and column subset of $A$. Transposes and pseudoinverses are taken after the subscripts, e.g., $A_{R:}^T=(A_{R:})^T$. 

\subsection{Two principles}\label{sec:twoprinciples}
Our proofs are built on two principles. The first principle is well-known in compressed sensing, and dates back to Rudelson \cite{rudelson1999random} in 1999. Intuitively, it says:
\begin{center}\medskip\fbox {\parbox{0.9\linewidth}{
Let $Y$ be a $n\times k$ matrix with orthonormal columns. Let $Y_{C:}$ be a \emph{random} row subset of $Y$. Suppose $Y$ is $\mu$-coherent with $\mu=\widetilde{O}(1)$, and $\abs{C}=\l \gtrsim \mu k $. Then with high probability, $(n/\l)^{1/2} Y_{C:}$ is like an isometry.
}}\medskip
\end{center}
To be precise, we quote \cite[Lemma 3.4]{tropp2010improved}. (Note that their $M$ is our $\mu k$.)
\begin{theorem}\label{thm:spreadSRIP}
Let $Y\in \cplexes^{n\times k}$ with orthonormal columns. Suppose $Y$ is $\mu$-coherent and $\l \geq \alpha k \mu $ for some $\alpha>0$. Let $Y_{C:}$ be a random $\l$-row subset of $Y$. Each row of $Y_{C:}$ is sampled independently, uniformly. Then
$$\P{\norm{Y_{C:}^+}\geq \sqrt{\frac{n}{(1-\delta)\l}} } \leq k \left(\frac{e^{-\delta}}{(1-\delta)^{1-\delta}} \right)^{\alpha } \text{ for any } \delta\in[0,1)$$
and
$$\P{\norm{Y_{C:}}\geq \sqrt{\frac{(1+\delta') \l }{n}}}\leq k \left(\frac{e^{\delta'}}{(1+\delta')^{1+\delta'}} \right)^{\alpha } \text{ for any } \delta'\geq 0.$$
\end{theorem}
If $\delta=0.57$ and $\delta'=0.709$ and $\l \geq 10 k \mu \log n$, then 
\begin{equation}\label{eq:firstprincipleest}
\P{\norm{Y_{C:}^+}\leq 1.53(n/\l)^{1/2} \text{ and }\norm{Y_{C:}}\leq 1.31(\l/n)^{1/2}}\geq 1-2 k n^{-2}.
\end{equation}
We will use (\ref{eq:firstprincipleest}) later. Let us proceed to the second principle, which says
\begin{center}
\medskip\fbox {\parbox{0.9\linewidth}{
Let $C$ be a \emph{nonrandom} index set. If $\norm{A_{:C}}$ is small , then $\norm{A}$ is also small, provided that we have control over $\norm{A Y_2}$ and $\norm{Y_{1,C:}^+}$ for some unitary matrix $(Y_1 \quad Y_2)$.
}}\medskip
\end{center}

The motivation is as follows. If we ignore the regularization step, then what we want to show is that $A\simeq A_{:C} A_{RC}^+ A_{R:}$. But when we take row and column restrictions on both sides, we have trivially $A_{RC} = A_{RC} A_{RC}^+ A_{RC}$. Hence, we desire a mechanism to go backwards, that is to infer that ``$E:=A-A_{:C} A_{RC}^+ A_{R:}$ is small'' from ``$E_{RC}$ is small.'' We begin by inferring that ``$E$ is small'' from ``$E_{:C}$ is small''.
\begin{lemma}\label{thm:lift}
Let $A\in \cplexes^{m\times n}$ and $Y=(Y_1 \quad Y_2)\in \cplexes^{n \times n}$ be a unitary matrix such that $Y_1$ has $k$ columns. Select $\l \geq k$ rows of $Y_1$ to form $Y_{1,C:}=S_C^T Y_1\in \cplexes^{\l \times k}$. Assume $Y_{1,C:}$ has full column rank. Then
$$\norm{A} \leq \norm{Y_{1,C:}^+} \norm{A_{:C}}+ \norm{Y_{1,C:}^+} \norm{A Y_2 Y_{2,C:}^T} + \norm{A Y_2}.$$
\end{lemma}
\begin{proof}
Note that $Y_{1,C:}^T Y_{1,C:}^{T+}=I_{k\times k}$. Now,
\begin{align*}
\norm{A }&\leq \norm{A Y_1} + \norm{A Y_2 }\\
&= \norm{A Y_1 Y_{1,C:}^T Y_{1,C:}^{T+} } +\norm{A Y_2}\\
&\leq \norm{A Y_1 Y_1^T S_C}\norm{Y_{1,C:}^+} + \norm{A Y_2}\\
& \leq \norm{ (A-A Y_2 Y_2^T)S_C}\norm{Y_{1,C:}^{+}} + \norm{A Y_2}\\
& \leq \norm{A_{:C}}\norm{Y_{1,C:}^+} + \norm{A Y_2 Y_{2,C:}^T} \norm{Y_{1,C:}^+} + \norm{A Y_2}.
\end{align*}
\end{proof}

Lemma \ref{thm:lift} can be extended in two obvious ways. One, we can deduce that ``$A$ is small if $A_{R:}$ is small.'' Two, we can deduce that ``$A$ is small if $A_{RC}$ is small.'' This is what the next lemma says. The proof is a straightforward modification of the proof of Lemma \ref{thm:lift} and shall be omitted.

\begin{lemma}\label{thm:liftRC}
Let $A\in \cplexes^{m\times n}$ and $X=(X_1\quad X_2)\in \cplexes^{m\times m}$ and $Y=(Y_1 \quad Y_2) \in \cplexes^{n\times n}$ be unitary matrices such that $X_1,Y_1$ each has $k$ columns. Select $\l\geq k$ rows and columns indexed by $R,C$ respectively. Assume $X_{1,R:},Y_{1,C:}$ have full column rank. Then
$$\norm{A} \leq \norm{X_{1,R:}^+} \norm{A_{R:}} + \norm{X_{1,R:}^+} \norm{X_{2R} X_2^T A}+ \norm{X_2^T A}$$
and
\begin{align*}
\norm{A} \leq & \norm{X_{1,R:}^+} \norm{Y_{1,C:}^+}\norm{A_{RC}}+\\
&\norm{X_{1,R:}^+} \norm{Y_{1,C:}^+}\norm{X_{2,R:} X_2^T A Y_1 Y_{1,C:}^T }+\\
&\norm{X_{1,R:}^+} \norm{Y_{1,C:}^+}\norm{ X_{1,R:} X_1^T A Y_2 Y_{2,C:}^T }+\\
&\norm{X_{1,R:}^+} \norm{Y_{1,C:}^+}\norm{X_{2,R:} X_2^T A Y_2 Y_{2,C:}^T }+\\
& \norm{X_{1,R:}^+}\norm{X_{1,R:} X_1^T A Y_2 }+\\
&\norm{Y_{1,C:}^+} \norm{X_2^T A Y_1 Y_{1,C:}^T}+\\
&\norm{X_2^T A Y_2}.
\end{align*}
\end{lemma}

We conclude this section with a useful corollary. It says that if $P A_{:C}$ is a good low rank approximation of $A_{:C}$ for some $P\in \cplexes^{m\times m}$, then $P A$ may also be a good low rank approximation of $A$.

\begin{corollary}\label{thm:liftPA}
Let $A\in \cplexes^{m\times n}$ and $P\in \cplexes^{m\times m}$. Let $Y=(Y_1 \quad Y_2)$ be a unitary matrix such that $Y_1$ has $k$ columns. Let $Y_{1,C:}=S_C^T Y_1\in \cplexes^{\l \times k}$ where $\l\geq k$. Assume $Y_{1,C:}$ has full column rank. Let $I\in \cplexes^{m\times m}$ be the identity. Then
\begin{equation*}\label{eq:corP}
\norm{A -P A} \leq \norm{Y_{1,C:}^+}\norm{A_{:C} - P A_{:C}} + \norm{Y_{1,C:}^+}\norm{I-P}\norm{A Y_2 Y_{2,C:}^T } + \norm{I-P}\norm{A Y_2}.
\end{equation*}
In particular, if $P$ is the orthogonal projection $A_{:C} A_{:C}^+$, then
\begin{equation}\label{eq:liftPA}
\norm{A-A_{:C} A_{:C}^+ A} \leq \norm{Y_{1,C:}^+}\norm{A Y_2 Y_{2,C:}^T}  + \norm{A Y_2}.
\end{equation}
\end{corollary}
\begin{proof}
To get the first inequality, apply Lemma \ref{thm:lift} to $A-PA$. The second inequality is immediate from the first inequality since $\norm{A_{:C} - A_{:C} A_{:C}^+ A_{:C}}=0$.
\end{proof}

For the special case where $X,Y$ are singular vectors of $A$,  (\ref{eq:liftPA}) can be proved using the fact that $\norm{A-A_{:C} A_{:C}^+ A} = \min_{D} \norm{A-A_{:C} D}$ and choosing an appropriate $D$. See \cite{boutsidis2009improved}.

 (\ref{eq:liftPA}) can be strengthened to $\norm{A-A_{:C} A_{:C}^+ A}^2 \leq \norm{A Y_2 Y_{2,C:}^T  Y_{1,C:}^+}^2 + \norm{AY_2}^2$, by modifying the first step of the proof of Lemma \ref{thm:lift} from $\norm{A}\leq \norm{A Y_1}+\norm{A Y_2}$ to $\norm{A}^2 \leq \norm{A Y_1}^2 + \norm{A Y_2}^2$. A similar result for the case where $X,Y$ are singular vectors can be found in \cite{halkofinding}. The value of our results is that it works for a more general model (\ref{eq:A}).

\subsection{Proof of Theorem \ref{thm:fullspread-k3}}\label{sec:proofk3}
Let $\lambda_X=(m/\l)^{1/2}$ and $\lambda_Y=(n/\l)^{1/2}$. To prove the first part of Theorem \ref{thm:fullspread-k3}, that is (\ref{eq:basicbound}), apply the first principle or Theorem \ref{thm:spreadSRIP}. From (\ref{eq:firstprincipleest}), it is clear that $\norm{Y_{1,C:}}=O(\lambda_Y^{-1})$, $\norm{Y_{1,C:}^+}=O(\lambda_Y)$, $\norm{X_{1,R:}}=O( \lambda_X^{-1})$, $\norm{X_{1,R:}^+}=O( \lambda_X)$ hold simultaneously with probability at least $1-4km^{-2}$.

For the second part of Theorem \ref{thm:fullspread-k3}, that is (\ref{eq:extrabound}), we need to refine (\ref{eq:A}) as follows. Let $X=(\widetilde{X}_1, \ldots, \widetilde{X}_{\lceil m/k \rceil})$ and $Y=(\widetilde{Y}_1,\ldots,\widetilde{Y}_{\lceil n/k \rceil})$ where $\widetilde{X}_1,\ldots,\widetilde{X}_{\lceil m/k \rceil-1}$ and $\widetilde{Y}_1,\ldots,\widetilde{Y}_{\lceil n/k \rceil-1}$ has $k$ columns, and $\widetilde{X}_{\lceil m/k \rceil},\widetilde{Y}_{\lceil n/k \rceil}$ have $\leq k$ columns. Note that $\widetilde{X}_1=X_1$, $\widetilde{Y}_1=Y_1$, $\widetilde{A}_{11}=A_{11}$ where $X_1,Y_1,A_{11}$ are defined in  (\ref{eq:A}). Rewrite  (\ref{eq:A}) as
\begin{equation}\label{eq:A2}
A=(\widetilde{X}_1, \ldots, \widetilde{X}_{\lceil m/k \rceil})\left( \begin{array}{ccc}
\widetilde{A}_{11}  & \ldots &  \widetilde{A}_{1,\lceil n/k \rceil} \\
\vdots & \ddots & \vdots\\
\widetilde{A}_{\lceil m/k \rceil, 1} & \ldots &  \widetilde{A}_{\lceil m/k \rceil,\lceil n/k \rceil}
\end{array} \right)
\left( \begin{array}{c} \widetilde{Y}_1^T \\ \vdots \\ \widetilde{Y}_{\lceil n/k \rceil}^T \end{array} \right).
\end{equation}
By applying Theorem \ref{thm:spreadSRIP} to every $\widetilde{X}_i$, $\widetilde{Y}_j$ and doing a union bound, we see that with probability at least $1-4m^{-1}$, we will have $\norm{Y_{j,C:}}=O(\lambda_Y^{-1})$, $\norm{Y_{j,C:}^+}=O(\lambda_Y)$, $\norm{X_{i,R:}}=O( \lambda_X^{-1})$, $\norm{X_{i,R:}^+}=O( \lambda_X)$ for all $i,j$.

The rest of the proof is \emph{deterministic}.

Suppose the SVD of $A_{RC}$ is $U \diag(\sigma) V^T$. Define another vector $\sigma'$ such that $\sigma'_i = -\sigma_i$ if $\sigma_i<\delta$, zero otherwise. Let $F=U \diag(\sigma') V^T$. The skeleton decomposition we return is $A_{:C} (A_{RC}+F)^+ A_{R:}$. The point of adding $F$ to $A_{RC}$ is to avoid ``inverting'' singular values that are less than $\delta$.

Define the $\l\times \l$ matrix $E$ such that $E_{RC}=F$ and all other entries are zeros. Now, let $B=A+E$ so that $B_{RC}=A_{RC}+F$. The skeleton returned is $A_{C:} B_{RC}^+ A_{R:}$. By construction, $\norm{A-B}\leq \delta$ and $\norm{B_{RC}^+}\leq \delta^{-1}$.

Our objective is to bound $\norm{A-A_{:C} B_{RC}^+ A_{R:}}$. Split this up by a perturbation argument.
\begin{align}\label{eq:firststep}
\norm{A -A_{:C} B_{RC}^+ A_{R:}} \leq & \norm{A-B}+\norm{B-B_{:C} B_{RC}^+ B_{R:}} + \nonumber\\
& \norm{B_{:C} B_{RC}^+ B_{R:} - A_{:C} B_{RC}^+ B_{R:}}+\norm{A_{:C} B_{RC}^+ B_{R:}-A_{:C} B_{RC}^+ A_{R:}}\nonumber\\
\leq&  \delta + \norm{B-B_{:C} B_{RC}^+ B_{R:}}+\nonumber \\ 
&\norm{(B-A)S_C}\norm{ B_{RC}^+ B_{R:}}+\norm{A_{:C}B_{RC}^+}\norm{S_R^T(B-A)}\nonumber\\
\leq&  \delta + \norm{B-B_{:C} B_{RC}^+ B_{R:}} + \nonumber\\
&\delta \norm{B_{RC}^+ B_{R:}} + ( \norm{B_{:C} B_{RC}^+}+\norm{A_{:C} B_{RC}^+ -B_{:C} B_{RC}^+} )\delta\nonumber\\
\leq& \delta + \norm{B-B_{:C} B_{RC}^+ B_{R:}}+\nonumber \\
&\delta  \norm{B_{RC}^+ B_{R:}}+ \delta \norm{B_{:C} B_{RC}^+}+ \norm{(A-B)S_C} \delta^{-1} \delta \nonumber \\
\leq&  2\delta + \norm{B-B_{:C} B_{RC}^+ B_{R:}}+ \delta \norm{B_{RC}^+ B_{R:}}+\delta \norm{B_{:C} B_{RC}^+}
\end{align}

It remains to bound $\norm{B-B_{:C} B_{RC}^+ B_{R:}},\norm{B_{RC}^+ B_{R:}},\norm{B_{:C} B_{RC}^+}$ using the second principle. Intuitively, $\norm{B_{RC}^+ B_{R:}}$ cannot be too much bigger than $\norm{B_{RC}^+ B_{RC}}\leq 1$.

By Lemma \ref{thm:lift},
\begin{align*}
\norm{B_{RC}^+ B_{R:}} \leq &\norm{Y_{1,C:}^+} \norm{B_{RC}^+ B_{RC}}+ \norm{Y_{1,C:}^+}\norm{B_{RC}^+ B_{R:} Y_2 Y_{2,C:}^T} + \norm{B_{RC}^+ B_{R:} Y_2}\\
\leq  & \norm{Y_{1,C:}^+}+ \norm{Y_{1,C:}^+}\norm{B_{RC}^+}( \norm{(B_{R:}  -A_{R:})Y_2 Y_{2,C:}^T }+ \norm{A_{R:}Y_2 Y_{2,C:}^T }) +\\
& \norm{B_{RC}^+}(\norm{ (B_{R:}-A_{R:}) Y_2}+\norm{A_{R:} Y_2})\\
\leq &   \norm{Y_{1,C:}^+}+ \norm{Y_{1,C:}^+}\delta^{-1}(\delta+\norm{A_{R:}Y_2 Y_{2,C:}^T }) + \delta^{-1}(\delta+\norm{A_{R:}Y_2 })\\
\leq &  1+2\norm{Y_{1,C:}^+}+ \norm{Y_{1,C:}^+}\delta^{-1}\norm{A_{R:}Y_2 Y_{2,C:}^T }+\delta^{-1}\norm{A_{R:}Y_2 }. 
\end{align*}
By the first principle, the following holds whp:
\begin{equation}\label{eq:BRCBR}
\norm{B_{RC}^+ B_{R:}} =O(\lambda_Y+ \lambda_Y \delta^{-1} \norm{A_{R:} Y_2 Y_{2,C:}^T} +\delta^{-1} \norm{A_{R:} Y_2}).
\end{equation}
The same argument works for $\norm{B_{:C} B_{RC}^+}$. Whp,
\begin{equation}\label{eq:BCBRC}
\norm{B_{:C} B_{RC}^+ }=O(\lambda_X+ \lambda_X \delta^{-1} \norm{X_{2,R:} X_2^T A_{:C}}+ \delta^{-1} \norm{X_2^T A_{:C}})
\end{equation}

Bounding $\norm{B-B_{:C} B_{RC}^+ B_{R:}}$ requires more work, but the basic ideas are the same. Recall that the second principle suggests that $\norm{B-B_{:C} B_{RC}^+ B_{R:}}$ cannot be too much bigger than $\norm{B_{RC}-B_{RC} B_{RC}^+ B_{RC}}=0$. Applying Lemma \ref{thm:liftRC} with $A$ being $B-B_{:C} B_{RC}^+ B_{R:}$ yields
\begin{align*}
\norm{B-B_{:C} B_{RC}^+ B_{R:}}\leq & \norm{X_{1,R:}^+} \norm{Y_{1,C:}^+} \norm{X_{2,R:} X_2^T (B-B_{:C} B_{RC}^+ B_{R:}) Y_1 Y_{1,C:}^T}+\\
&\norm{X_{1,R:}^+} \norm{Y_{1,C:}^+} \norm{X_{1,R:} X_1^T (B-B_{:C} B_{RC}^+ B_{R:}) Y_2 Y_{2,:C}^T}+\\
&\norm{X_{1,R:}^+} \norm{Y_{1,C:}^+}\norm{X_{2,R:} X_2^T (B-B_{:C} B_{RC}^+ B_{R:})Y_2 Y_{2,C:}^T}+\\
&\norm{X_{1,R:}^+}\norm{X_{1,R:} X_1^T (B-B_{:C} B_{RC}^+ B_{R:}) Y_2}+\\
& \norm{Y_{1,C:}^+}\norm{X_2^T (B-B_{:C} B_{RC}^+ B_{R:})Y_1 Y_{1,C:}^T}+\\
&\norm{X_2^T (B-B_{:C} B_{RC}^+ B_{R:}) Y_2}
\end{align*}
which is bounded by
\begin{align*}&\norm{X_{1,R:}^+} \norm{Y_{1,C:}^+}\norm{Y_{1,C:}}(\norm{X_{2,R:} X_2^T B}+ \norm{X_{2,R:} X_2^T B_{:C}}\norm{B_{RC}^+ B_{R:}})+\\
& \norm{X_{1,R:}^+} \norm{Y_{1,C:}^+} \norm{X_{1,R:}}(\norm{B Y_2 Y_{2,C:}^T}+\norm{B_{R:} Y_2 Y_{2,C:}^T}\norm{B_{:C} B_{RC}^+})\\
& \norm{X_{1,R:}^+} \norm{Y_{1,C:}^+}(\norm{X_{2,R:} X_2^T B Y_2 Y_{2,C:}^T}+ \norm{X_{2,R:} X_2^T B_{:C}} \delta^{-1} \norm{B_{R:} Y_2 Y_{2,C:}^T})+\\
& \norm{X_{1,R:}^+} \norm{X_{1,R:}} (\norm{B Y_2}+ \norm{B_{:C} B_{RC}^+}\norm{B_{R:} Y_2})+\\
& \norm{Y_{1,C:}^+}\norm{Y_{1,C:}}(\norm{X_2^T B}+ \norm{B_{RC}^+ B_{R:}}\norm{X_2^T B_{:C}})+\\
& \norm{X_2^T B Y_2}+ \norm{X_2^T B_{:C}}\delta^{-1} \norm{B_{R:} Y_2}.
\end{align*}
We have paired $\norm{X_{1,R:}}$ with $\norm{X_{1,R:}^+}$, and $\norm{Y_{1,C:}}$ with $\norm{Y_{1,C:}^+}$ because the first principle 
implies that their products are $O(1)$ whp. That means
\begin{align*}\norm{B-B_{:C} B_{RC}^+ B_{R:}}= O(&\lambda_X(\norm{X_{2,R:} X_2^T B}+ \norm{X_{2,R:} X_2^T B_{:C}}\norm{B_{RC}^+ B_{R:}})+\\
&\lambda_Y (\norm{B Y_2 Y_{2,C:}^T}+\norm{B_{R:} Y_2 Y_{2,C:}^T}\norm{B_{:C} B_{RC}^+})+\\
& \lambda_X \lambda_Y (\norm{X_{2,R:} X_2^T B Y_2 Y_{2,C:}^T}+ \norm{X_{2,R:} X_2^T B_{:C}} \delta^{-1} \norm{B_{R:} Y_2 Y_{2,C:}^T})+\\
&\norm{B Y_2}+ \norm{B_{:C} B_{RC}^+}\norm{B_{R:} Y_2}+\\
& \norm{X_2^T B}+ \norm{B_{RC}^+ B_{R:}}\norm{X_2^T B_{:C}}+\\
&  \norm{X_2^T B_{:C}}\delta^{-1} \norm{B_{R:} Y_2}).
\end{align*}
We have dropped  $\norm{X_2^T B Y_2}$ because it is dominated by $\norm{X_2^T B}$.  (\ref{eq:BCBRC}) and  (\ref{eq:BRCBR}) can be used to control $\norm{B_{:C}B_{RC}^+}$ and $\norm{B_{RC}^+ B_{R:}}$. Before doing that, we want to replace $B$ with $A$ in all the other terms. This will introduce some extra $\delta$'s. For example, $\norm{X_{2,R:} X_2^T B}\leq \norm{X_{2,R:} X_2^T B - X_{2,R:} X_2^T A}+\norm{X_{2,R:} X_2^T A}\leq \delta+\norm{X_{2,R:} X_2^T A}$. Doing the same for other terms, we have that $\norm{B-B_{:C} B_{RC}^+ B_{R:}}$ is whp
\begin{align*} O(&\lambda_X(\delta+\norm{X_{2,R:} X_2^T A}+ (\delta+\norm{X_{2,R:} X_2^T A_{:C}})\norm{B_{RC}^+ B_{R:}})+\\
&\lambda_Y (\delta+\norm{A Y_2 Y_{2,C:}^T}+(\delta +\norm{A_{R:} Y_2 Y_{2,C:}^T})\norm{B_{:C} B_{RC}^+})+\\
& \lambda_X \lambda_Y (\delta+\norm{X_{2,R:} X_2^T A Y_2 Y_{2,C:}^T}+ \norm{X_{2,R:} X_2^T A_{:C}}+\\
& \norm{A_{R:} Y_2 Y_{2,C:}^T}+ \norm{X_{2,R:} X_2^T A_{:C}} \delta^{-1} \norm{A_{R:} Y_2 Y_{2,C:}^T})+\\
&\delta+\norm{A Y_2}+ (\delta +\norm{A_{R:} Y_2})\norm{B_{:C} B_{RC}^+}+\\
&\delta+\norm{X_2^T A}+ (\delta +\norm{X_2^T A_{:C}}) \norm{B_{RC}^+ B_{R:}}+\\
& \delta+ \norm{X_2^T A_{:C}}+\norm{A_{R:} Y_2}+\norm{X_2^T A_{:C}}\delta^{-1} \norm{A_{R:} Y_2}).
\end{align*}
Several terms can be dropped in the $O$ notation, for example $\delta\leq \lambda_X \delta \leq \lambda_X \lambda_Y \delta$ and $\norm{X_{2}^T A_{:C}}\leq \norm{X_2^T A}$. We shall also plug in the estimates on $\norm{B_{:C}B_{RC}^+}$ and $\norm{B_{RC}^+ B_{R:}}$,  from  (\ref{eq:BCBRC}) and  (\ref{eq:BRCBR}). This leads to
\begin{align*} O(&\lambda_X(\norm{X_{2,R:} X_2^T A}+ (\delta+\norm{X_{2,R:} X_2^T A_{:C}})(\lambda_Y+\lambda_Y \delta^{-1} \norm{A_{R:} Y_2 Y_{2,C:}^T}+\delta^{-1} \norm{A_{R:} Y_2}))+\\
&\lambda_Y (\norm{A Y_2 Y_{2,C:}^T}+(\delta +\norm{A_{R:} Y_2 Y_{2,C:}^T})(\lambda_X+\lambda_X \delta^{-1} \norm{X_{2,R:}X_2^T A_{:C}}+\delta^{-1} \norm{X_2^T A_{:C}}))+\\
& \lambda_X \lambda_Y (\delta+\norm{X_{2,R:} X_2^T A Y_2 Y_{2,C:}^T}+ \norm{X_{2,R:} X_2^T A_{:C}}+\\
& \norm{A_{R:} Y_2 Y_{2,C:}^T}+ \norm{X_{2,R:} X_2^T A_{:C}} \delta^{-1} \norm{A_{R:} Y_2 Y_{2,C:}^T})+\\
&\norm{A Y_2}+ (\delta +\norm{A_{R:} Y_2})(\lambda_X+\lambda_X \delta^{-1} \norm{X_{2,R:}X_2^T A_{:C}}+\delta^{-1} \norm{X_2^T A_{:C}})+\\
&\norm{X_2^T A}+ (\delta +\norm{X_2^T A_{:C}}) (\lambda_Y+\lambda_Y \delta^{-1} \norm{A_{R:} Y_2 Y_{2,C:}^T}+\delta^{-1} \norm{A_{R:} Y_2})+\\
& \norm{X_2^T A_{:C}}\delta^{-1} \norm{A_{R:} Y_2}).
\end{align*}
Collect the terms by their $\lambda_X, \lambda_Y$ factors and drop the smaller terms to obtain that whp,  is
\begin{align}\label{eq:BCBRCBR}
\norm{B-B_{:C} B_{RC}^+ B_{R:}}=O(&\lambda_X(\norm{X_{2,R:} X_2^T A}+ \norm{A_{R:} Y_2}+\delta^{-1}\norm{X_{2,R:}X_2^T A_{:C}}\norm{A_{R:} Y_2} )+\nonumber\\
& \lambda_Y(\norm{A Y_2 Y_{2,C:}^T}+ \norm{X_2^T A_{:C}}+\delta^{-1}\norm{A_{R:} Y_2 Y_{2,C:}^T}\norm{X_2^T A_{:C}} )+\nonumber\\
& \lambda_X \lambda_Y(\delta+  \norm{X_{2,R:}X_2^T A_{:C}}+\norm{A_{R:} Y_2 Y_{2,C:}^T}+\nonumber\\
&\delta^{-1}\norm{X_{2,R:}X_2^T A_{:C}}\norm{A_{R:} Y_2 Y_{2,C:}^T}+\norm{X_{2,R:}X_2^T A Y_2 Y_{2,C:}^T})+\nonumber\\
&\norm{X_2^T A}+\norm{A Y_2}+\delta^{-1} \norm{X_2^T A_{:C}}\norm{A_{R:}Y_2}).
\end{align}
We now have control over all three terms $\norm{B-B_{:C} B_{RC}^+ B_{R:}},\norm{B_{RC}^+ B_{R:}},\norm{B_{:C} B_{RC}^+}$. Substitute (\ref{eq:BRCBR}), (\ref{eq:BCBRC}), (\ref{eq:BCBRCBR}) into  (\ref{eq:firststep}). As the RHS of  (\ref{eq:BCBRCBR}) dominates $\delta$ multiplied by the RHS of  (\ref{eq:BRCBR}), (\ref{eq:BCBRC}), we conclude that whp, $\norm{A-A_{:C} B_{RC}^+ A_{R:}}$ is also bounded by the RHS of  (\ref{eq:BCBRCBR}).

To obtain the basic bound,  (\ref{eq:basicbound}), we note that all the ``normed terms'' on the RHS of (\ref{eq:BCBRCBR}), e.g., $\norm{A_{R:} Y_2 Y_{2,C}^T}$ and $\norm{X_2^T A}$, are bounded by $\eps_k$. It follows that whp,  $\norm{A-A_{:C} B_{RC}^+ A_{R:}}=O(\lambda_X \lambda_Y (\delta+\eps+\delta^{-1} \eps^2))$.

To obtain the other bound,  (\ref{eq:extrabound}), we need to bound each ``normed term'' of (\ref{eq:BCBRCBR}) differently.

Recall (\ref{eq:A2}). Consider $\norm{X_{2,R:} X_2^T A_{:C}}$. See that
$$X_{2,R:} X_2^T A_{:C} =(\widetilde{X}_{2,R:}, \ldots, \widetilde{X}_{\lceil m/k \rceil,R:}) \left( \begin{array}{ccc}
\widetilde{A}_{21}  & \ldots &  \widetilde{A}_{2,\lceil n/k \rceil} \\
\vdots & \ddots & \vdots\\
\widetilde{A}_{\lceil m/k \rceil, 1} & \ldots &  \widetilde{A}_{\lceil m/k \rceil,\lceil n/k \rceil}\end{array} \right)
\left( \begin{array}{c} \widetilde{Y}_{1,C:}^T \\ \vdots \\ \widetilde{Y}_{\lceil n/k \rceil,C:}^T \end{array} \right).$$
Recall that at the beginning of the subsection, we show that whp, $\norm{\widetilde{X}_{i,R:}}=O(\lambda_X^{-1})$ and $\norm{\widetilde{Y}_{j,C:}}=O(\lambda_Y^{-1})$ for all $i,j$. Recall the definition of $\eps'_k$ in  (\ref{eq:eps1}). It follows that whp,
$$\norm{X_{2,R:} X_2^T A_{:C}}\leq \sum_{i=2}^{\lceil m/k \rceil}\sum_{j=1}^{\lceil n/k \rceil} \norm{\widetilde{X}_{i,R:}}\norm{\widetilde{A}_{ij}}\norm{\widetilde{Y}_{j,C:}} \leq \lambda_X^{-1}\lambda_{Y}^{-1} \eps'_k.$$
Apply the same argument to other terms on the RHS of  (\ref{eq:BCBRCBR}), e.g., $\norm{X_{2,R:} X_2^T A Y_2 Y_{2,C:}^T} =O(\lambda_X^{-1} \lambda_Y^{-1} \eps'_k)$ and $\norm{X_2^T A_{:C}}=O(\lambda_Y^{-1} \eps'_k)$ whp. Mnemonically, a $R$ in the subscript leads to a $\lambda_X^{-1}$ and a $C$ in the subscript leads to a $\lambda_Y^{-1}$.

Recall that $\norm{A_{:C} B_{RC}^+ A_{R:}}$ is bounded by the RHS of  (\ref{eq:BCBRCBR}). Upon simplifying, we obtain that $\norm{A-A_{:C} B_{RC}^+ A_{R:}} = O(\lambda_X \lambda_Y \delta +\eps'_k + \lambda_X \lambda_Y \delta^{-1} {\eps'_k}^2)$, that is (\ref{eq:extrabound}). The proof is complete.

\section{Two other algorithms}\label{sec:extend}

\subsection{Second algorithm}
Algorithm \ref{alg:fullspread-slow} returns a skeleton with exactly $k$ rows and columns. First, randomly select $\widetilde{O}(k)$ rows and columns, then trim down to $k$ columns and $k$ rows by performing RRQR on $A_{:C}$ and $A_{R:}^T$. For dense matrices, the most expensive step is the multiplication of $A$ by $A_{R':}^+$. However, for structured matrices, the most expensive step of Algorithm \ref{alg:fullspread-slow} may be the RRQR factorization of $A_{:C}$ and $A_{R:}^T$ and the inversion of $A_{:C'},A_{R':}$, which take $\widetilde{O}(m k^2)$ time. The overall running time is $O(T_A k)+\widetilde{O}(m k^2)$ where $T_A$ is the cost of matrix-vector multiplication, but for convenience, we refer to this as the $\widetilde{O}(mnk)$ algorithm.

\begin{figure}[t] \begin{center}\medskip\fbox {\parbox{0.9\linewidth}{
\begin{algorithm}\label{alg:fullspread-slow}
\textup{ $\widetilde{O}(mnk)$ algorithm \\
\underline{Input}: Matrix $A\in \cplexes^{m\times n}$.\\
\underline{Output}: A skeleton representation with \emph{exactly} $k$ rows and columns.\\
\underline{Steps}:
\begin{enumerate}
\item Uniformly and independently select $\l$ columns to form $A_{:C}$.
\item Uniformly and independently select $\l$ rows to form $A_{R:}$.
\item Run RRQR on $A_{:C}$ to select $k$ columns and form $A_{:C'}$. This takes $\widetilde{O}(m k^2)$ time and $\widetilde{O}(m k)$ space.
\item Run RRQR on $A_{R:}^T$ to select $k$ rows and form $A_{R':}$. This takes $\widetilde{O}(n k^2)$ time and $\widetilde{O}(nk)$ space.
\item Compute $Z=A_{:C'}^+ (A A_{R':}^+)$. This takes $O(T_A k + m k^2)$ time and $O(m k)$ space, where $T_A$ is the time needed to apply $A$ to a vector.
\item Output $A_{:C'} Z A_{R':}$.
\end{enumerate}
} \end{algorithm}
}}\medskip \end{center} 
\end{figure}

It can be shown that once $A_{:C'},A_{R':}$ are fixed, the choice of $Z=A_{:C'}^+ A A_{R':}^+$ is optimal in the Frobenius norm (but not in the $2$-norm), that is $Z=\arg_{W\in \cplexes^{\l\times \l}} \norm{A-A_{:C'} W A_{R':}}_F$. Not surprisingly, we have a better error estimate.

\begin{theorem}\label{thm:fullspread-slow}
Let $A$ be given by  (\ref{eq:A}) for some $k > 0$. Assume $m\geq n$ and $X_1,Y_1$ are $\mu$-coherent. Recall the definitions of $\eps_k, \eps'_k$ in  (\ref{eq:eps}) and (\ref{eq:eps1}). Let $\l \geq 10\mu k \log m$. With probability at least $1-4km^{-2}$, Algorithm \ref{alg:fullspread-slow} returns a skeleton that satisfies
$$\norm{A-A_{:C'} Z A_{R':}} =O(\eps_k (m k)^{1/2}).$$
\end{theorem}
\begin{proof}
Let $P=A_{:C'} A_{:C'}^+ \in \cplexes^{m \times m}$. RRQR \cite{gu1996efficient} selects $k$ columns from $A_{:C}$ such that
\begin{equation*}\label{eq:RRQRcond1}
\norm{A_{:C}- P A_{:C}} \leq f(k,\l)\sigma_{k+1}(A_{:C}) \leq  f(k,\l) \sigma_{k+1}(A) \leq f(k,\l) \eps_k,
\end{equation*}
where $f(k,\l):=\sqrt{1+2k(\l-k)}$. We have used the fact that $\sigma_{k+1}(A) = \sigma_{k+1}\twotwo{A_{11}}{A_{12}}{A_{21}}{A_{22}}\leq \sigma_{1} \binom{A_{12}}{A_{22}}\leq \eps_k$.
See interlacing theorems \cite[Corollary 3.1.3]{horn1994topics}.

Recall from (\ref{eq:firstprincipleest}) that $\norm{Y_{1,C:}^+} =O((n/\l)^{1/2})$ with probability at least $1-2km^{-2}$. Apply Corollary \ref{thm:liftPA} to obtain that whp 
\begin{align*}
\norm{A - P A}&\leq O(\lambda_Y) \norm{A_{:C}-P A_{:C}} + O(\lambda_Y) \eps_k+ \eps_k\\
& =O( \eps_k (n/\l)^{1/2} f(k,\l))=O(\eps_k (nk)^{1/2}).
\end{align*}
Let $P'=A_{R':}^+ A_{R':}$. By the same argument, $\norm{A-A P'}=O(\eps_k(mk)^{1/2})$ with the same failure probability. Combine both estimates. With probability at least $1-4km^{-2}$,
\begin{align*}
\norm{A-A_{:C'} A_{:C'}^+ A A_{R':}^+ A_{R':}} &=\norm{A-P A P'}\\
&\leq \norm{A-PA}+\norm{PA-PAP'}\\
&\leq \norm{A-PA}+\norm{A-AP'}\\
&=O(\eps_k (mk)^{1/2}).
\end{align*}
\end{proof}

Many algorithms that use the skeleton $A_{:C} (A_{:C}^+ A A_{R:}^+) A_{R:}$, e.g. \cite{mahoney2009cur}, seek to select columns such that $\norm{A-A_{:C} A_{:C}^+ A}$ is small. Here, we further select $k$ out of $\l=\widetilde{O}(k)$ columns, which is also suggested in \cite{boutsidis2009improved}. Their $2$-norm error estimate is $O(k \log^{1/2} k) \eps_k + O(k^{3/4} \log^{1/4} k) \|A-A_k\|_F$ where $A_k$ is the optimal rank $k$ approximation to $A$.  In general $\norm{A-A_k}_F \leq (n-k)^{1/2}\eps_k$, so our bound is a factor $k^{1/4}$ better. Our proof is also more straightforward. Nevertheless, we make the extra assumption that $X_1,Y_1$ are incoherent.

\subsection{Third algorithm}
Consider the case where \emph{only} $X_1$ is $\widetilde{O}(1)$-coherent. See Algorithm \ref{alg:partialspread-nk2}. It computes a skeleton with $\widetilde{O}(k)$ rows and $k$ columns in $\widetilde{O}(nk^2+k^3)$ time. Intuitively, the algorithm works as follows. We want to select $k$ columns of $A$ but running RRQR on $A$ is too expensive. Instead, we randomly choose $\widetilde{O}(k)$ rows to form $A_{R:}$, and select our $k$ columns using the much smaller matrix $A_{R:}$. This works because $X_1$ is assumed to be $\widetilde{O}(1)$-coherent and choosing almost any $\widetilde{O}(k)$ rows will give us a good sketch of $A$.

\begin{figure}[t] \begin{center}\medskip\fbox {\parbox{0.9\linewidth}{
\begin{algorithm}\label{alg:partialspread-nk2}
\textup{\underline{$\widetilde{O}(nk^2)$ algorithm}\\
\underline{Input}: Matrix $A\in \cplexes^{m\times n}$.\\
\underline{Output}: A skeleton representation with $\l=\widetilde{O}(k)$ rows and $k$ columns.\\
\underline{Steps}:
\begin{enumerate}
\item Uniformly and independently select $\l$ rows to form $A_{R:}$.
\item Run RRQR on $A_{R:}$. Obtain $A_{R:} \simeq A_{RC'}(A_{RC'}^+ A_{R:})$  where $A_{RC'}$ contains $k$ columns of $A_{R:}$. This takes $\widetilde{O}(n k^2)$ time and $\widetilde{O}(nk)$ space.
\item Compute $Z=A_{RC'}^+$. This takes $\widetilde{O}(k^3)$ time and $\widetilde{O}(k^2)$ space.
\item Output $A_{:C'} Z A_{R:}$.
\end{enumerate}
}\end{algorithm}
}}\medskip \end{center}
\end{figure}

\begin{theorem}\label{thm:partialspread-nk2}
Let $A$ be given by  (\ref{eq:A}) for some $k > 0$. Assume $m\geq n$ and $Y_1$ is $\mu$-coherent. (There is no assumption on $X_1$.) Recall the definition of $\eps_k$ in  (\ref{eq:eps}). Let $\l \geq 10\mu k \log m$. Then, with probability at least $1-2km^{-2}$, Algorithm \ref{alg:partialspread-nk2} returns a skeleton that satisfies
$$\norm{A-A_{:C'} Z A_{R:}} =O(\eps_k (mn)^{1/2}).$$
\end{theorem}
\begin{proof}
We perform RRQR on $A_{R:}$ to obtain $A_{R:} \simeq A_{RC'}D$ where $D=A_{R C'}^+ A_{R:} $ and $C'$ indexes the selected $k$ columns. We want to use the second principle to ``undo the row restriction'' and infer that $A\simeq A_{C'}D$, the output of Algorithm \ref{alg:partialspread-nk2}. We now fill in the details.

Strong RRQR \cite{gu1996efficient} guarantees that
$$\norm{A_{R:} - A_{R C'}D} \leq \sigma_{k+1}(A_{R:}) f(k,n)  \leq \sigma_{k+1}(A) f(k,n)  \leq \eps_k f(k,n) $$
and
$$\norm{D} \leq f(k,n)$$
where $f(k,n) = \sqrt{1+2k(n-k)}$.
Prepare to apply a transposed version of Corollary \ref{thm:liftPA}, that is 
\begin{equation}\label{eq:corP2}
\norm{A-AP} \leq \norm{X_{1,R:}^+}\norm{A_{R:}-A_{R:} P} + \norm{X_{1,R:}^+}\norm{I-P} \norm{X_{2,R:} X_2^T A} +\norm{I-P}\norm{X_2^T A}.
\end{equation}
Let $P=S_{C'} D $, so that $\norm{P}\leq \norm{D}\leq f(k,n)$. Note that $AP = A_{:C'} A_{R C'}^+ A_{R:}$. By (\ref{eq:firstprincipleest}), with probability at least $1-2km^{-2}$, $\norm{X_{1,R:}^+}=O((m/\l)^{1/2})$. By (\ref{eq:corP2}),
\begin{align*}
\norm{A -AP}&\leq O(\lambda_X) \norm{A_{R:}-A_{RC'}D }+O(\lambda_X)(1+\norm{P}) \eps_k + (1+\norm{P})\eps_k\\
& =O(\eps_k f(k,n) (m/\l)^{1/2}) =O( \eps_k(mn)^{1/2}).
\end{align*}
\end{proof}

If $X_1$ is not incoherent and we fix it by multiplying on the left by a randomized Fourier matrix $\mathcal{F} D$ (cf. Section \ref{sec:introsrft}), then we arrive at the algorithm in \cite{liberty2007randomized}. The linear algebraic part of their proof combined with the first principle will lead to the \emph{same bounds}. What we have done here is to decouple the proof into three simple parts: (1) show that $\widetilde{X}_1:=\mathcal{F} D X_1$ is incoherent, (2) use the first principle to show that $\widetilde{X}_{1,R:}$ is ``sufficiently nonsingular'', (3) use the second principle to finish up the linear algebra.

\subsection{Comparing the three algorithms}\label{sec:summary}
Here is a summary of the three algorithms studied in this paper. Assume $m\geq n$.
\begin{table}[h]
\begin{center}
\begin{tabular}{|c|c|c|c|c|c|} \hline
& No. rows & No. columns & Error in $2$-norm & Running time & Memory\\ \hline
Algorithm 1 & $\l=\widetilde{O}(k)$ & $\l=\widetilde{O}(k)$ & $O(\eps_k(mn)^{1/2}/\l) $ & $\widetilde{O}(k^3)$ & $\widetilde{O}(k^2)$\\
Algorithm 2 & $k$ & $k$ & $O(\eps_k(mk)^{1/2})$ & $O(T_A k)+\widetilde{O}(mk^2)$ & $\widetilde{O}(mk)$ \\
Algorithm 3 & $\l=\widetilde{O}(k)$ & $k$ & $O(\eps_k(mn)^{1/2})$ & $\widetilde{O}(nk^2)$ & $\widetilde{O}(nk)$ \\ \hline
\end{tabular}
\end{center}
\end{table}

Suppose $A$ has structure and $T_A \ll mn$. If we can tolerate a $m$ factor in the running time and memory usage, then Algorithm 2 is the method of choice. Otherwise, we recommend using Algorithm 1. It is much faster and even though Theorem \ref{thm:fullspread-k3} suggests that the error grows with $(mn)^{1/2}$, we believe that in practice, the error usually grows with $m^{1/2}$. 

Suppose $A$ is dense and has no structure. In this case, Algorithm 2 is too slow. As for Algorithm 3, it runs significantly slower than Algorithm 1 and its error bounds are not any better. However, if we do not want to worry about setting the parameter $\delta$, we will prefer to use Algorithm 3.

\section{Examples}\label{sec:numerical}
For this section, we consider only Algorithm  \ref{alg:fullspread-k3}, the $\widetilde{O}(k^3)$ algorithm.

\subsection{Toy example}\label{sec:toy}
Let $A=X \Sigma Y^T \in \cplexes^{n \times n}$ where $X,Y$ are unitary \emph{Fourier matrices} and $\Sigma$ is a diagonal matrix of singular values. Note that every entry of $X,Y$ is of magnitude $n^{-1/2}$ and are $1$-coherent.

For the first experiment, $A$ is $301\times 301$ and $\eps=\eps_{k}=\sigma_{k+1}=\ldots=\sigma_n=10^{-15}$. We also pick the first $k$ singular values to be logarithmically spaced between $1$ and $\eps_k$. In each trial, we randomly pick $\l=100$ rows and columns and measure $\norm{A-A_{:C} Z A_{R:}}$. The only parameters being varied are $k$ and $\delta$.

\begin{figure}[t]
\begin{center}\includegraphics{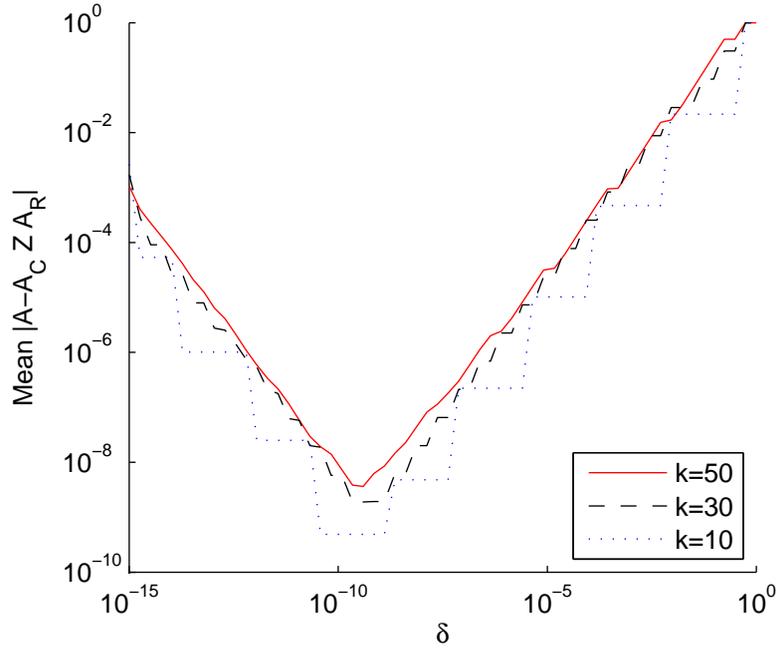}\end{center}
\caption{\label{fig:exp11_k3_delta_plot} Loglog plot of the empirical mean of the error (in $2$-norm) by the $\widetilde{O}(k^3)$ algorithm versus $\delta$, a regularization parameter. This relationship between the error and $\delta$ agrees with Theorem \ref{thm:fullspread-k3}. See  (\ref{eq:logV}). More importantly, the error blows up for small $\delta$, which implies that the regularization step is essential.}
\end{figure}

From  (\ref{eq:basicbound}) in Theorem \ref{thm:fullspread-k3}, we may expect that when variables such as $n,m,\l,k$ are fixed,
\begin{equation}\label{eq:logV}
\log \norm{A-A_{:C} Z A_{R:}} \propto \log (\delta^{-1}(\eps_k+\delta)^{2}) = -\log \delta + 2\log(\eps_k+\delta).
\end{equation}
Consider a plot of $\norm{A-A_{:C} Z A_{R:}}$ versus $\delta$ on a log-log scale. According to the above equation, when $\delta\ll \eps_k$, the first term dominates and we expect to see a line of slope $-1$, and when $\delta\gg \eps_k$, $\log (\eps_k + \delta) \simeq \log \delta $ and we expect to see a line of slope $+1$. Indeed, when we plot the experimental results in Figure \ref{fig:exp11_k3_delta_plot}, we see a right-angled $V$-curve.

A curious feature of Figure \ref{fig:exp11_k3_delta_plot} is that the error curves behaves like a staircase. As we decrease $k$, the number of different error levels decrease proportionally. A possible explanation for this behavior is that the top singular vectors of $A_{:C}$ match those of $A$, and as $\delta$ increases from $\sigma_i(A)$ to $\sigma_{i-1}(A)$ for some small $i$, smaller components will not be inverted and the error is all on the order of $\sigma_i(A)$.

For the second experiment, we use the same matrix but vary $n$ instead of $\delta$. We fix $k=9$, $\l=40$ and $\delta=\eps_{k}\l/n^{1/2}$. There are three different runs with $\eps_k=\eps=10^{-6},10^{-8},10^{-10}$.  The results are plotted in Figure \ref{fig:exp4_k3_plot}. They suggest that the error $\norm{A-A_{:C} Z A_{R:}}$ scales with $n^{1/2}$, which is better than (\ref{eq:basicconsequence}).

\begin{figure}[t]
\begin{center}\includegraphics{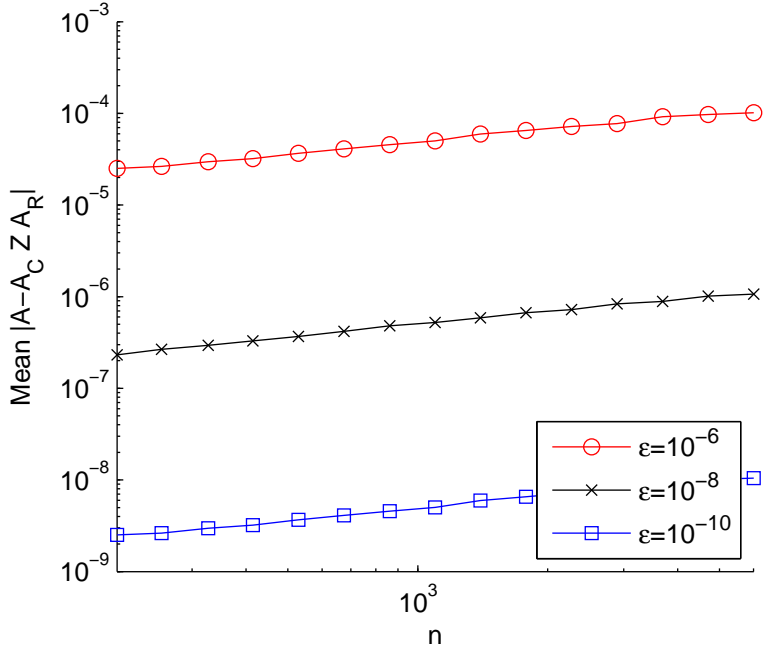}\end{center}
\caption{\label{fig:exp4_k3_plot} Loglog plot of the empirical mean of the error (in $2$-norm) by the $\widetilde{O}(k^3)$ algorithm versus $n$, the size of square matrix $A$. Here, $k,\l$ are fixed, $A=X_1 Y_1^T + \eps X_2 Y_2^T$ and $\eps_k=\eps_{k+1}=\ldots=\eps$. The parameter $\delta$ is set at $\eps(\l/n^{1/2})$. The error is roughly proportional to $n^{1/2}\eps$.}
\end{figure}

\subsection{Smooth kernel}\label{sec:smoothkernel}

Consider a 1D integral operator with a kernel $K$ that is analytic on $[-1,1]^2$. Define $A$ as $(A)_{ij}=c K(x_i,y_j)$ where the nodes $x_1,\ldots,x_n$ and $y_1,\ldots,y_n$ are \emph{uniformly spaced} in $[-1,1]$. First, suppose $K= \sum_{1\leq i,j\leq 6} c_{ij} T_i(x)T_j(y)+ 10^{-3} T_{10}(x) T_{10}(y)+ 10^{-9}N$ where $T_i(x)$ is the $i$-th Chebyshev polynomial and $N$ is the random Gaussian matrix, i.e., noise. The coefficients $c_{ij}$'s are chosen such that $\norm{A} \simeq 1$. Pick $n=m=10^3$ and slowly increase $\l$, the number of rows and columns sampled by the $\widetilde{O}(k^3)$ algorithm. As shown in Figure \ref{fig:exp50_plot}, the skeleton representation $A_{:C} Z A_{R:}$ converges rapidly to $A$ as we increase $\l$.

\begin{figure}[t]
\begin{center}\includegraphics{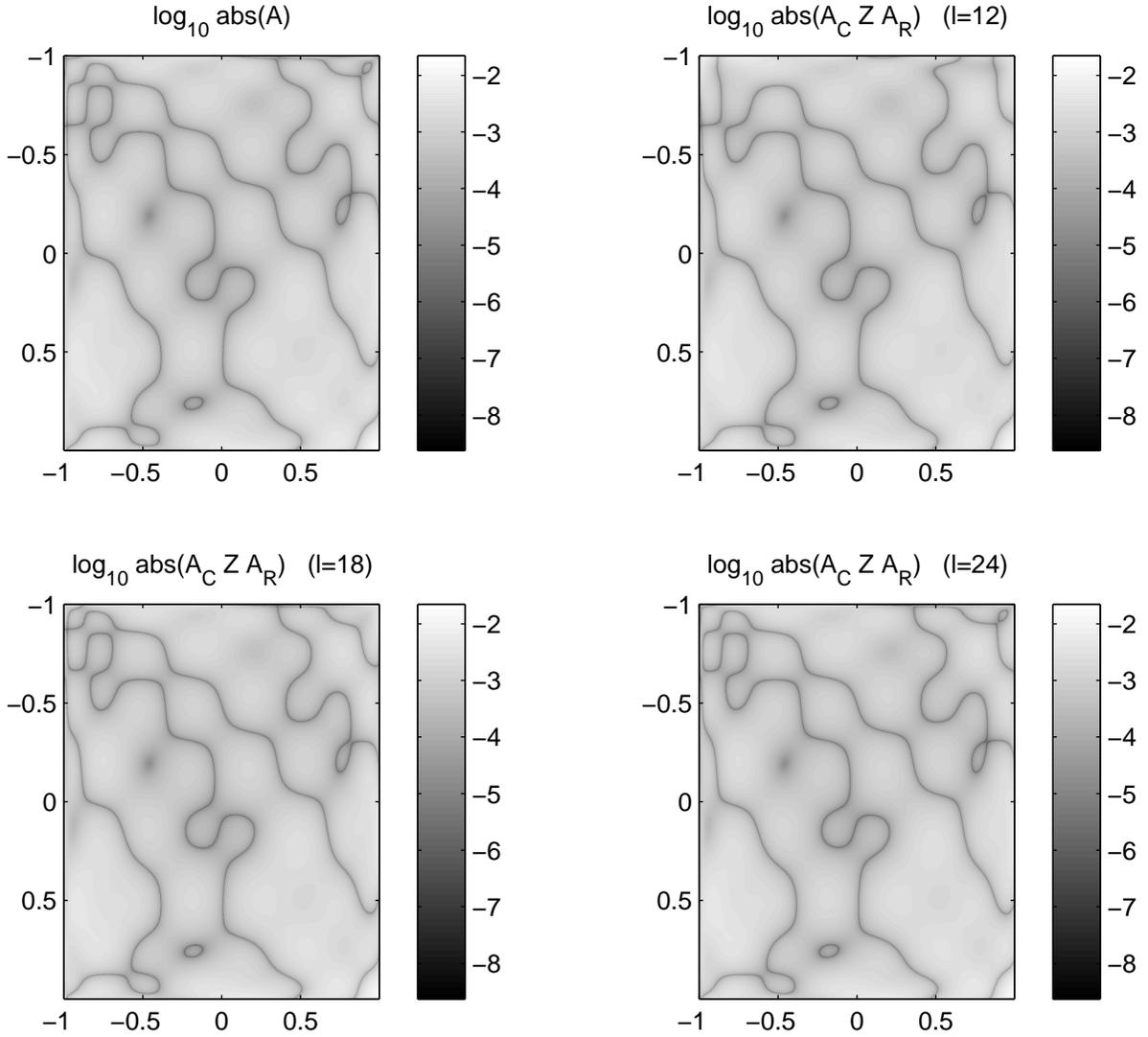}\end{center}
\caption{\label{fig:exp50_plot} $A$ is the smooth kernel $K(x,y)$ where $K$ is the sum of $6^2$ low degree Chebyshev polynomials evaluated on a $10^3 \times 10^3$ uniform grid. The topleft figure is $A$ while the other figures show that the more intricate features of $A$ start to appear as we increase $\l$ from $12$ to $18$ to $24$. Recall that we sample $\l$ rows and $\l$ columns in the $\widetilde{O}(k^3)$ algorithm.}
\end{figure}

Next, consider $K(x,y)=c\exp(xy)$. Let $n=900$ and pick $c$ to normalize $\norm{A} = 1$. We then plot the empirical mean of the error of the $\widetilde{O}(k^3)$ algorithm against $\l$ on the left of Figure \ref{fig:exp103_plot}. Notice that the error decreases exponentially with $\l$.

\begin{figure}[t]
\begin{center}\includegraphics{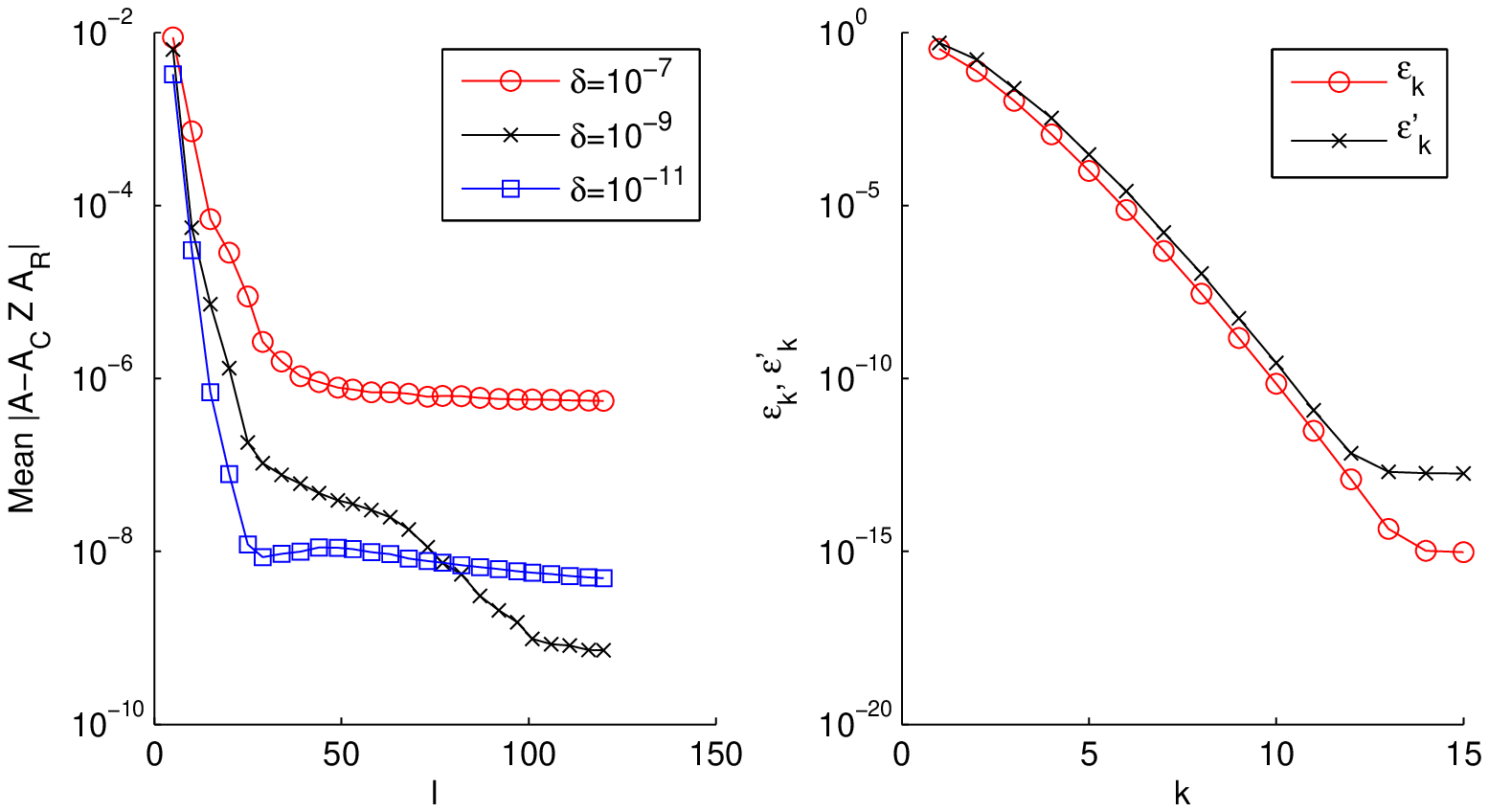}\end{center}
\caption{\label{fig:exp103_plot} $A$ is the smooth kernel $K(x,y)=\exp(-xy)$ sampled on a uniform grid. The graph on the left shows that the error of the $\widetilde{O}(k^3)$ algorithm decreases exponentially with $\l$, the number of sampled rows and columns. The figure on the right shows that if we expand $A$ in terms of Legendre polynomials, the coefficients (and therefore $\eps_k,\eps'_k$) decay exponentially. See  (\ref{eq:A}),  (\ref{eq:eps}) and  (\ref{eq:eps1}) for the definitions of $\eps_k$ and $\eps'_k$.}
\end{figure}

To understand what is happening, imagine that the grid is infinitely fine. Let $\vphi_1,\vphi_2,\ldots$ be {Legendre polynomials}. Recall that these polynomials are orthogonal on $[-1,1]$. Define the matrices $X,Y$ as $(X)_{ij}=\vphi_j(x_i)$ and $(Y)_{ij}=\vphi_j(y_i)$. Assume the $\vphi_i$'s are scaled such that $X,Y$ are unitary. It is well-known that if we expand $K$ in terms of Chebyshev polynomials or Legendre polynomials \cite{boyd2001chebyshev} or prolate spheroidal wave functions \cite{xiao2001prolate}, the expansion coefficients will decay exponentially. This means that the entries of $X^T A Y$ should decay exponentially away from the topleft corner and $\eps'_k = \Theta(\eps_k)$ (cf.  (\ref{eq:eps}) and  (\ref{eq:eps1})). We confirm this by plotting $\eps_k,\eps'_k$ versus $k$ on the right of Figure \ref{fig:exp103_plot}. The actual $X,Y$ used to obtain this plot are obtained by evaluating the Legendre polynomials on the uniform grid and orthonormalizing. It can be verified that the entries of $X,Y$ are of magnitude $O((k/n)^{1/2})$ which implies a coherence $\mu\simeq k$, independent of $n$. The implication is that the algorithm will continue to perform well as $n$ increases.

As $\l$ increases, we can apply Theorem \ref{thm:fullspread-k3} with a larger $k$. Since $\eps_k,\eps'_k$ decrease exponentially, the error should also decrease exponentially. However, as $k$ increases beyond  $\simeq 15$, $\eps_k$ stagnates and nothing can be gained from increasing $\l$. In general, as $\eps_k$ decreases, we should pick a smaller $\delta$. But when $k\gtrsim 15$, choosing a smaller $\delta$ does not help and may lead to worse results due to the instability of pseudoinverses and floating point errors. This is evident from Figure \ref{fig:exp103_plot}.

Platte, Trefethen, Kuijlaars \cite{platte2010impossibility} state that if we sample on a uniform grid, the error of {any stable} approximation scheme cannot decrease exponentially forever.  In this example, the random selection of columns and rows correspond to selecting interpolation points randomly from a uniform grid, and $\delta$ serves as a regularization parameter of our approximation scheme. The method is stable, but we can only expect an exponential decrease of the error up to a limit dependent on $\delta$.

\subsection{Fourier integral operators}
In \cite{candes:fio}, Candes et al. consider how to efficiently apply 2D Fourier integral operators of the form
$$Lf(x) = \int_{\xi} a(x,\xi) e^{2\pi i \Phi(x,\xi)} \hat{f}(\xi) d\xi$$
where $\hat{f}(\xi)$ is the Fourier transform of $f$, $a(x,\xi)$ is a smooth amplitude function and $\Phi$ is a smooth phase function that is homogeneous, i.e., $\Phi(x,\lambda \xi) = \lambda \Phi(x,\xi)$ for any $\lambda>0$. Say there are $N^2$ points in the space domain and also the frequency domain.

The main idea is to split the frequency domain into $\sqrt{N}$ wedges, Taylor expand $\Phi(x,\cdot)$ about $\abs{\xi} \hat{\xi}_j$ where $j$ indexes a wedge, and observe that the residual phase $\Phi_j(x,\xi):=\Phi(x,\xi)-\Phi(x,\abs{\xi}\hat{\xi}_j)\cdot \xi$ is nonoscillatory. Hence, the matrix $A_{st}^{(j)}: = \exp(2\pi i \Phi_j(x_s,\xi_t))$ can be approximated by a low rank matrix, i.e., $\exp(2\pi i \Phi_j(x,\xi))$ can be written as $\sum_{q=1}^{r} f_q(x) g_q(\xi)$ where $r$, the separation rank, is independent of $N$. By switching order of summations, the authors arrive at $\widetilde{O}(N^{2.5})$ algorithms for both the preprocessing and the evaluation steps. See \cite{candes:fio} for further details.

What we are concerned here is the approximate factorization of $A^{(j)}$. This is a $N^2$ by $N^{1.5}$ matrix since there are $N^2$ points in the space domain and $N^2/\sqrt{N}$ points in one wedge in the frequency domain. In \cite{candes:fio}, a slightly different algorithm is proposed.
\begin{enumerate}
\item Uniformly and randomly select $\l$ rows and columns to form $A_{R:}$ and $A_{:C}$.
\item Perform SVD on $A_{:C}$. Say $A_{:C} = U_1 \Lambda_1 V_1^T+U_2 \Lambda_2 V_2^T$ where $U,V$ are unitary and $\norm{\Lambda_2}\leq \delta$, a user specified parameter.
\item Return the low rank representation $U_1 U_{1,R:}^+ A_{R:}$.
\end{enumerate}
In the words of the authors, ``this randomized approach works well in practice although we are not able to offer a rigorous proof of its accuracy, and expect one to be non-trivial'' \cite{candes:fio}. We believe this is because $A$ does satisfy the assumptions of this paper. See (\ref{eq:A}), (\ref{eq:eps}).

To see why, let $B$ be a perturbation of $A$ such that $B_{:C}=U_1 \Lambda_1 V_1^T$ and $\norm{A-B}\leq \delta$. Since $\Lambda_1$ is invertible, the output can be rewritten as
$$U_1 U_{1R}^+ A_{R:} = B_{:C}  B_{RC}^+ A_{R:}.$$

It is easy to adapt the proof of Theorem \ref{thm:fullspread-k3} to bound $\norm{A-B_{:C} B_{RC}^+ A_{R:}}$ --- just replace (\ref{eq:firststep}) with the following. The rest of the proof is the same.
\begin{align*}
\norm{A -B_{:C} B_{RC}^+ A_{R:}} &\leq \norm{A-B}+\norm{B-B_{:C} B_{RC}^+ B_{R:}} + \norm{B_{:C} B_{RC}^+ B_{R:} - B_{:C} B_{RC}^+ A_{R:}}\\
&\leq  \delta + \norm{B-B_{:C} B_{RC}^+ B_{R:}}+\norm{B_{:C}B_{RC}^+}\delta.
\end{align*}

An important subclass of Fourier integral operators is pseudodifferential operators. These are linear operators with pseudodifferential symbols that obey certain smoothness conditions \cite{shubin01}. In Discrete Symbol Calculus \cite{demanet2011dsc}, a similar randomized algorithm is used to derive low rank factorizations of such smooth symbols. It is likely that the method works well here in the same way as it works well for a smooth kernel as discussed in the previous section.

\bibliography{CUR}
\bibliographystyle{siam}
\end{document}